\documentclass[letterpaper,12pt]{amsart}
\newtheorem{theorem}{Theorem}[section]

\newtheorem{lemma}[theorem]{Lemma}
\newtheorem{proposition}[theorem]{Proposition}
\theoremstyle{definition}
\newtheorem{definition}[theorem]{Definition}
\theoremstyle{remark}
\newtheorem{remark}[theorem]{Remark}

\def \sgn {\operatorname{sgn}}

\begin{document}
\author{Will Sawin}
\address{Department of Mathematics \\ Columbia University \\ 2990 Broadway \\ New York, NY 10027, USA}
\title[Dynamical models for Liouville]{Dynamical models for Liouville and obstructions to further progress on sign patterns}

\maketitle

\begin{abstract} We define a class of dynamical systems by modifying a construction due to Tao, which includes certain Furstenburg limits arising from the Liouville function. Most recent progress on the Chowla conjectures and sign patterns of the M\"{o}bius and Liouville functions uses methods that apply to any dynamical system in this class. Hence dynamical systems in this class with anomalous local behavior present obstructions to further progress on these problems by the same techniques. We construct straightforward examples of dynamical systems in this class based on polynomial phases and calculate the resulting obstruction. This requires explicit bounds for the number of sign patterns arising in a certain way from polynomials, which is elementary but not completely trivial. \end{abstract}

\section{Introduction}

Matom\"{a}ki and Radziwi{\l}{\l} proved that the Liouville function cancels in almost intervals of a given length, as long as the length increases to infinity \cite[Theorem 1]{mrcancellation}. Using this, Matom\"{a}ki, Radziwi{\l}{\l}, and Tao showed certain results about sign patterns of the Liouville function \cite{mrtsigns}. Here a sign pattern is an element of $\{-1,1\}^k$ and we say it occurs at a point $n$ if it equals $(\lambda(n),\dots, \lambda(n+k-1))$. They show that every three-term sign pattern occurs with positive density and give a lower bound for the number of $k$-term sign patterns that occur with positive density as a function of $k$. These two papers have a very different flavor. \cite{mrcancellation} uses the multiplicativity of Liouville at large primes, making it closer to traditional analytic number theory, and is more conceptual. \cite{mrtsigns} uses only multiplicativity at small primes, is more combinatorial, and uses ad-hoc arguments that involve checking different cases. Later work of Tao \cite{ttwochowla} and Tao and Ter\"{a}v\"{a}inen \cite{ttodd} improves on \cite{mrtsigns} by using information-theoretic methods to prove the logarithmically averaged Chowla's conjecture for all correlations of the Liouville function over two or an odd number of points, which among other applications gives further results on sign patterns.

To make further progress using the methods of \cite{mrcancellation}, \cite{ttwochowla}, and \cite{ttodd}, new ideas are needed, but \cite{mrtsigns} gives the impression that it's possible that even more complicated arguments, perhaps too complicated to be found by humans or computers, could prove new cases with longer sign patterns. It would be desirable to show that this is not entirely true, and that progress beyond a certain point on the sign pattern problem requires progress on the analytic side. The idea by Tao \cite{tfurstenburg} of studying the Furstenburg limits of Liouville gives a way to make this a well-defined mathematical problem, which we will solve in this paper. 

Given a sequence $\lambda(n)$ of signs, one can find a sequence $N_i$ of natural numbers such that for each $k$, and for each sign pattern of length $k$, the fraction of numbers from $1$ to $N_i$ where that sign pattern occurs converges as $i$ goes to $\infty$. Such a sequence determines a measurable dynamical system (in fact simply a measure on the space of infinite sequences of signs). The set of all such dynamical systems is called the set of Furstenburg limits of $\lambda$. For the Liouville function, it is possible to restrict to a subset of limits which preserve, in a certain sense, the multiplicativity property. This leads to the notion of limits of Liouville defined in an axiomatic way by Tao in \cite{tfurstenburg}. In \cite{tfurstenburg}, Tao notes that his work \cite{ttwochowla} can be deduced from \cite{mrcancellation} purely within the setting of limits of Liouville. In particular, it does not use the first axiom of \cite[Proposition 1]{tfurstenburg}, which states that the dynamical system actually arises as a limit of Liouville, except to apply the main theorem of \cite{mrcancellation}. The same is true for the combinatorial arguments of \cite{mrtsigns}. Hence we can test the boundaries of these methods by replacing that axiom by a weaker one, which encapsulates what has been proven on the analytic side (or a potential future analytic result). In this paper we will show concrete limits on what can be proven in these weaker axiom systems by explicit examples.

In fact we will use a slightly different set of axioms, which turns out to lead to simpler constructions of dynamical systems. As we will see below, dynamical systems satisfying all of Tao's axioms but the first can be constructed from dynamical models for $\lambda$ in our sense by an inverse limit construction.

\begin{definition} 
Let $\lambda: \mathbb N \to \{\pm 1\}$ be a multiplicative function. Define a \emph{dynamical model for $\lambda$} to be a measurable space $X$ with probability measure $\mu$, a measurable function $F:X \to \pm 1$ a measure-preserving transformation: $T: X \to X$, a measure preserving transformation $M: X \to \hat{\mathbb Z}$, and for each natural number $n$ a transformation $I_n: M^{-1} ( n \hat{\mathbb Z}) \to X$, such that that for any $x \in X$ outside a set of measure $0$:

\begin{enumerate}

\item For any $n\in \mathbb N$  with $M(x) \in n \hat{\mathbb Z}$, $n M (I_n(x)) = M(x)$.

\item $M (T (x) ) =M(x)+1$.

\item For any $n\in \mathbb N$ with $M(x) \in n \hat{\mathbb Z}$, $T(I_n(x))=I_n(T^n(x))$.

\item For any $n,m\in \mathbb N$ with $M(x)\in nm\hat{\mathbb Z}$, $I_m(I_n(x))=I_{nm}(x)$ (where the composition is well-defined by axiom (1)).

\item For any $n \in \mathbb N$ with $M(x) \in n\hat{\mathbb Z}$, $F(I_n(x))= \lambda(n) F(x)$. 

\end{enumerate}

and such that the pushforward of the measure $\mu$ along $I_n$ (after restriction to $M^{-1} (n\mathbb Z)$) is $\mu /n$. \end{definition}

\begin{remark} Here we think of $X$ as the natural numbers, $\mu$ as the uniform measure on some long interval, $F$ as $\lambda$, $T$ as increment, $M$ as the natural map $\mathbb N \to \hat{\mathbb Z}$, and $I_n$ as the division by $n$ map from multiples of $n$ to all numbers.

We can make this more precise by explaining how some dynamical models for $\lambda$ arise from Furstenburg limits of $\lambda$. We can take $X =( \pm 1)^{\mathbb N_0 } \times \hat{\mathbb Z}$. For each point $n \in \mathbb N$ we define a point $(h \mapsto \lambda(n+h), h)$ of $X$. By logarithmically averaging the point measures associated to numbers $n$ in the interval $[1,\dots,N]$, we can define a measure on $X$. Any sequence of $N$s admits a subsequence on which these measures converge to some limit, which we will take to be $\mu$. One takes $F(f,m)=f(0)$, $T((f,m)) =(h\mapsto f(h+1), m+1)$, $I_n ( f,m))= ( h\mapsto \lambda(n) f(nh), m/n)$ and verifies that identities (1-5) hold and that the measure associated to $N$ is approximately preserved by $T$, $M$, and $I_n$, with the last after multiplication by $1/2$, so the limit is exactly preserved by this transformation. Removing the $\hat{\mathbb Z}$ factor and projecting the measure down produces a Furstenburg system in the sense of \cite[Definition 2.4]{fhsystem}. In this way, positive results about dynamical models can be used to prove results about $\lambda$ itself, while dynamical models themselves give obstructions to this method. \end{remark}

\begin{definition} The sign pattern $\epsilon_1,\dots,\epsilon_k\in \{-1,1\}^k$ \emph{appears in $(X, \mu, F, T,M,I_n)$} if the measure of the set of $x \in X$ with $F( T^i (x))=\epsilon_i$ for $i$ from $1$ to $k$ is positive. \end{definition}

The definition of dynamical models of $\lambda$ we have given does not yet have an axiom encapsulating the analytic work. The result of Matom\"{a}ki and Radziwi{\l}{\l} implies that any Furstenburg limit of $\lambda$ has vanishing Gowers-Host-Kra $U^1$ norm and thus does not correlate with any $0$-step nilsystem. A plausible avenue for future progress on the analytic side is showing a vanishing $U^{d}$ norm, or equivalently that $\lambda$ does not correlate with any $d-1$-step nilsystem. For this reason, we focus our attention on the following definition.

\begin{definition} $(X, \mu, F, T, M,I_n)$ is \emph{$d$-Fourier uniform} if the Gowers-Host-Kra $U^d$ seminorm of $F$ on $(X, \mu, F, T)$ vanishes, i.e. \[ \lim_{H \to \infty} \mathbb E_{1 \leq h_1,\dots, h_d \leq H} \mathbb E_{ x \in X} \prod_{ \delta_1,\dots, \delta_d \in \{0,1\}}  F( T^{ \sum_{i=1}^d \delta_i h_i} x)  =0 \] and the same holds with any congruence conditions on $h_1,\dots,h_d, M(x)$. \end{definition} 

Notably, it should be possible to check, by an argument along the lines of \cite[Section 3]{tchowlasarnak}, that for $\lambda$ the Liouville function, if $(X, \mu, F, T,M,I_n)$ is $d$-Fourier uniform then it satisfies the Chowla conjecture for $d+1$-point correlations - see \cite[Remark 3.4]{tchowlasarnak}. By an argument along the lines of \cite{ttodd}, it should be possible to show the odd order Chowla conjectures in this setting without any uniformity assumption.

We will construct in this paper explicit dynamical models for $\lambda$ which have interesting properties with respect to sign patterns but have vanishing $U^d$ norm, for the simple reason that they are sufficiently general $d$-step nilsystems. 

\begin{theorem} For each multiplicative function $\lambda: \mathbb N \to \pm 1$ and natural number $d$, there exists a $d$-Fourier uniform dynamical model for $\lambda$, $(X,\mu, T, M,I_n)$, such that the number of distinct sign patterns of length $k$ that appear in $(X,\mu, T, M,I_n)$  is $O(k^{(d+1)(d+2)/2})$. \end{theorem}

\begin{theorem} For each multiplicative function $\lambda: \mathbb N \to \pm 1$, natural number $d$, and sign pattern $\epsilon_1 ,\dots,\epsilon_k$ such that there exists $q \leq k$ with $\phi(q) \geq \left( \frac{(d+1) (d+2) }{2}+ 2\right) \log_2 (q) + \log_2(8/3)$, there exists a dynamical model for $\lambda$ that is $d$-Fourier uniform but in which the sign pattern $\epsilon_1,\dots,\epsilon_k$ does not appear.

\end{theorem}

For instance, if $d=1$ we can take $q=29$, so for each $29$-term sign pattern we can find a $1$-uniform dynamical model in which it does not appear. Obviously, there is still a large gap between this construction and the positive results that are known. For fixed values of $d$ one can improve this by greater care with the explicit constants in our argument. However, it does not seem possible to completely close the gap in this way, nor to significantly improve the asymptotics of these bounds as $d$ goes to $\infty$. Another avenue of progress would be to improve the positive results using more tools from dynamical systems theory, in particular the classification of nilsystems.

Our dynamical systems will be created from polynomials of degree $d$. (This is unsurprising, as they are the simplest example of a dynamical system which is $d$-Fourier uniform but not $d+1$-Fourier uniform.) Hence the first step of our work will be a precise count of the number of sign patterns arising from polynomials of degree $d$.

This research was conducted during the period I served as a Clay Research Fellow, during the period I was supported by Dr. Max R\"{o}ssler, the Walter Haefner Foundation and the ETH Zurich Foundation, and during the period I was resident at MSRI. I would like to thank Terence Tao and Joni Ter\"{a}v\"{a}inen for helpful conversations.

\section{Polynomial Sign Patterns}

View $\mathbb R^{d+1}$ as the space of polynomials in one variable of degree $d$.

Let $\sgn: \mathbb R^{d+1} \to \{+1,-1\}^{k}$ be the function that sends a polynomial $f$ to the $k$-tuple of signs whose $n$th element is $(-1)^{\lfloor f(n) \rfloor}$. 

The goal of this section is to bound explicitly the cardinality of the image of $\sgn$ from the space of all degree $d$ polynomials.

\begin{definition} Define a \emph{pre-cell} to be a closed convex subset of $\mathbb R^{d+1}$ of the form $\{f \in \mathbb R^{d+1} | v_n \leq f(n) \leq v_{n}+1 \textrm{ for all }n\textrm{ from }1\textrm{ to }k\}$ for a $k$-tuple of integers $v_1,\dots,v_k$. Say that it is a \emph{cell} if it has nonempty interior. 

Say that two nontrivial cells $C_1,C_2$ are equivalent if $C_2$ is equal to the translation of $C_1$ by a polynomial $f \in \mathbb R^{d+1}$ such that  $f(n) \in 2 \mathbb Z$ for all $n$ in $\mathbb Z$.

\end{definition}

\begin{lemma}\label{open} A polynomial $f$ lies in the interior of a precell if and only if it lies in the precell and $f(n)$ is not an integer for any $n$ from $1$ to $k$. \end{lemma}

\begin{proof} If $f$ lies in the precell $\{f \in \mathbb R^{d+1} | v_n \leq f(n) \leq v_{n}+1 \textrm{ for all }n\textrm{ from }1\textrm{ to }k\}$ and takes non-integer values, then $f$ lies in $\{f \in \mathbb R^{d+1} | v_n < f(n) < v_{n}+1 \textrm{ for all }n\textrm{ from }1\textrm{ to }k\}$ which is open and is contained in the precell, and hence $f$ lies in the interior.

If $f$ takes some integer value, then for some $n \in \{1,\dots,k\}$, $f(n)=v_n$ or $f(n)=v_{n+1}$. In the first case $f-\epsilon$ is not in the precell for any positive $\epsilon$, and in the second case $f(n)+\epsilon$ is not in the precell for any positive $\epsilon$, so in neither case is $f$ in the interior.\end{proof}

\begin{lemma}\label{perturbation} Every sign pattern in the image of $\sgn$ is attained for some polynomial whose coefficients lie in the interior of some cell. \end{lemma}

\begin{proof} Let $g(x)$ be a polynomial of degree $d$ and let $v(n) = \lfloor g(n) \rfloor$. Consider the precell $\{f \in \mathbb R^{d+1} | v_n \leq f(n) \leq v_{n}+1 \textrm{ for all }n\textrm{ from }1\textrm{ to }k\}$. 

For $\epsilon>0$ sufficiently small, we have $g(n)+ \epsilon \in (\lfloor g(n)\rfloor, \lfloor g(n) \rfloor +1)$.

Hence $\lfloor g(n)+ \epsilon \rfloor = \lfloor g(n) \rfloor$, so $\sgn(g(n)+\epsilon)=\sgn(g(n)$, and $g+ \epsilon \in [v_n,v_{n+1}]$, so $g+\epsilon$ lies in this precell, and finally $g(n)+\epsilon$ takes noninteger values, so it lies in the interior of the precell by Lemma \ref{open}, and thus the precell is in fact a cell.\end{proof}

\begin{lemma}\label{constant} $\sgn$ is constant on the interior of each cell. \end{lemma}

\begin{proof} The function $(-1)^{\lfloor f(n) \rfloor}$ is locally constant everywhere that $f(n)$ is not an integer. For $n$ from $1$ to $k$, $f(n)$ is not an integer on the interior of each cell by Lemma \ref{open}. Because cells are convex bodies, their interiors are connected, so this function is constant on the interior of each cell. Hence $\sgn$ is constant on the interior of each cell. \end{proof}

\begin{lemma}\label{equivalent} On two equivalent cells, $\sgn(f)$ takes the same value. \end{lemma}

\begin{proof} $\sgn(f)$ is invariant under translation by even-integer-valued polynomials, so it takes the same value on two sets where one is a translation of another by such a polynomial. \end{proof}

\begin{lemma}\label{image} The cardinality of the image of $\sgn(f)$ is at most the number of equivalence classes of nontrivial cells. \end{lemma}

\begin{proof} Every sign pattern in the image of $f$ is attained in the interior of some nontrivial cell by Lemma \ref{perturbation}. Because the sign pattern is constant on each cell by Lemma \ref{constant}, and equal for equivalent cells by Lemma \ref{equivalent}, there is at most one sign pattern appearing in each equivalence class of cells. \end{proof}

\begin{lemma} Any vertex of any nontrivial cell is a polynomial $f$ that takes integer values on at least $d+1$ elements of $\{1,\dots,k\}$. \end{lemma}

\begin{proof} Suppose that some polynomial $f(x)$ takes non-integer values at all points in $\{1,\dots,k\}$ except $n_1,\dots,n_m$ for some $m \leq d$. We must show $f$ is not a vertex. Then for $|\epsilon|$ sufficiently small, $f(x) + \epsilon \prod_{i=1}^m (x-n_i)$ lies in the same cell as $f(x)$, and so $f(x)$ is in the convex hull of two points $f(x)+\epsilon \prod_{i=1}^m (x-n_i)$ and $f(x)-\epsilon \prod_{i=1}^m (x-n_i)$ of the same cell and thus is not a vertex. \end{proof}

\begin{lemma}\label{shelah} A polynomial $g(x)$ that takes integer values at exactly $m \geq d+1$ elements of $\{1,\dots,k\}$ is a vertex of at most $2 \sum_{i=0}^{d} {m-1 \choose i} \leq 2^{d+1} {m \choose d+1}$ nontrivial cells. \end{lemma}

\begin{proof} Let $S=\{n \in \mathbb N | 1 \leq n \leq k, g(n) \in \mathbb Z\}$.  Suppose that \[\{f\in \mathbb R^{d+1} | v_n \leq f(n) \leq v_{n}+1 \textrm{ for all }n\textrm{ from }1\textrm{ to }k\} \] is a precell containing $g$ as a vertex Then for any $n \in \{1,\dots,k\}-S$, we must have $v_n = \lfloor g(n) \rfloor$. At any $n \in S$, we have either $v_n = g(n)$ or $v_n =g(n)-1$. So there are $2^m$ precells containing $g$ as a vertex, determined by subsets $T \subseteq S$ where the precell corresponding to $T$ has $v_n=g(n)$ for $n \in T$ and $v_n=g(n)-1$ for $n \not \in T$. 

If the precell corresponding to $T$ is a cell, then by Lemma \ref{open} there is a polynomial $f$ such that $f(n) \in (g(n),g(n)+1)$ for $n \in T$ and $f(n) \in (g(n)-1,g(n))$ for $n \in S-T$. In particular, $f(n)-g(n)>0$ for $n \in T$ and $f(n)-g(n) < 0 $ for $n \in S-T$. Let the elements of $S$ be $n_1,\dots,n_m$. Then if $n_j \in T$ and $n_{ j+1} \in S-T$, or vice versa, $f(n)-g(n)$ has a root between $n_j$ and $n_{j+1}$, so there can only be $d$ such values of $j$. Hence the number of cells is at most the number of sign patterns on $\{1,\dots,m\}$ that only change signs at most $d$ times, which is \[2\sum_{i=0}^{d} {m-1 \choose i} \leq \sum_{i=0}^{d+1} {m \choose i} \leq  2^{d+1} {m \choose d+1}.\]

 \end{proof}

\begin{lemma}\label{translation} Suppose $f$ is a vertex of a cell $C$, and $g$ is a polynomial that takes even integer values at integer arguments. Then $f+g$ is a vertex of the equivalent cell $C+g$. \end{lemma}

\begin{proof} If $C$ is of the form $\{f \in \mathbb R^{d+1} | v_n \leq f(n) \leq v_{n}+1 \textrm{ for all }n\textrm{ from }1\textrm{ to }k\}$, we can take $v_{n'}= v_n+O(g_n)$ to define a precell $C+g$. Adding $g$ to an interior point of $C$ will produce an interior point of $C+g$. \end{proof}

\begin{lemma} Assume $k \geq d+1$. Then each cell has at least $d+2$ vertices. \end{lemma}

\begin{proof} The definition implies they are bounded closed convex bodies with nonempty interior in $\mathbb R^{d+1}$, hence have at least $d+2$ vertices. \end{proof}

\begin{lemma}\label{switching} The number of nontrivial cells, up to equivalence, is at most $2^{d+1}/(d+2)$ times the number of pairs of a polynomial $f(x)$ of degree $d$ and a subset $S$ of $[1,\dots,k]$ of cardinality $d+1$ on which $f(x)$ takes integer values, up to the equivalence relation where two pairs $(f_1,S_1),(f_2,S_2)$ are equivalent if $f_1-f_2$ takes only even integer values on integer arguments and $S_1=S_2$. \end{lemma}

\begin{proof}Consider the set of pairs of a cell and a vertex of that cell, and the equivalence relation where two pairs are equivalent if one is the translation of the other by a polynomial that takes even integer values on integer arguments. Each cell has at least $d+2$ vertices. Furthermore, for each cell the pairs of that cell and a vertex are in distinct equivalence classes, because no nontrivial translation sends the cell to itself and no trivial translation sends a vertex to a different vertex. So the number of such equivalence classes is at least $d+2$ times the number of equivalence classes of cells.

Let $f$ be a polynomial. For any polynomial $g$ that takes even integer values at integer arguments, every cell $C$ containing $f+g$ as a vertex is a translate of a cell containing $f$ as a vertex. So the number of equivalence classes of pairs of a cell and a vertex of that cell where the vertex is a translate of $f$ is at most the number of cells containing $f$ as a vertex, which by Lemma \ref{shelah} is at most $2^{d+1}$ times the number of $d+1$-element subsets of the points in $\{1,\dots,k\}$ where $f$ takes integer values.

Hence the total number of equivalence classes of pairs of a cell and a vertex of that cell is at most $2^{d+1}$ times the sum over equivalence classes of $f$ of the number of $d+1$-element subsets $S \subseteq \{1,\dots,k\}$ where $f$ has integer values, which is the number of equivalence classes of pairs as in the statement of the lemma. \end{proof}

\begin{lemma}\label{bhargava} For any $S$ of $d+1$ integers, the cardinality of the set of polynomials $f$ of degree $d$ that take integer values on $S$, up to translation by the set of polynomials that take even integer values on integer arguments, is \[ 2^{d+1}\frac{  \prod_{\substack{x,y \in S\\ x>y } } (x-y)   }{\prod_{n=1}^d n!} .\] \end{lemma}

\begin{proof}  Because both of these sets of polynomials are groups, equivalence classes are elements of the quotient group. So we must prove that the index of the group of polynomials that take even values on integers inside the group of polynomials that take integer values on $S$ is $2^{d+1}\frac{  \prod_{\substack{x,y \in S\\ x>y } } (x-y)   }{\prod_{n=1}^d n!}$. Because indices are multiplicative in iterated extensions, that follows from three facts:

First, the group of polynomials of degree $d$ that take integer values of integers  is a lattice of rank $d+1$, and the group taking even values on integers is twice that lattice, so the index of the group of polynomials taking even values on integers inside the group taking integer values on integers is $2^{d+1}$.

Second, the index of the group of polynomials with integer coefficients inside the group taking integer values on integers is $\prod_{n=1}^d n!$.

Third, the index of the group of polynomials with integer coefficients inside the group taking integer values on $S$ is the order of the cokernel of the map from polynomials with integer coefficients to integer-valued functions on $S$, which is the absolute value of the $d+1 \times d+1$ Vandermonde determinant associated to $S$, which is $\prod_{\substack{x,y \in S\\ x>y } } (x-y)  $.
 \end{proof}

\begin{theorem}\label{mainbound} The image of $\sgn(f)$ is at most \[ \frac{4^{d+1}}{d+2}\sum_{ \substack{ S \subseteq \{1,\dots,k\} \ |S| =d+1} }  \frac{ \prod_{\substack{x,y \in S\\ x>y }}(x-y) }{\prod_{n=1}^d n!}  \leq  c_d k^{ (d+2)(d+1)/2} \]

where  $c_1 = 8/3$ and $c_2=4/3$, and $c_d \leq 1$ for all $d \geq 3$. \end{theorem}

\begin{proof} By Lemma \ref{image}, this is at most the number of nontrivial cells up to equivalence. Thus by Lemma \ref{switching}, it is at most $\frac{2^{d+1}}{d+2}$ times the sum over sets $S$ of size $d+1$ of the number of polynomials $f$ that take integer values of $S$ up to the equivalence relation where $f_1$ and $f_2$ are equivalent if $f_1 -f_2$ takes even integer values on integer arguments. By Lemma \ref{bhargava}, this is \[  \sum_{\substack{ S \subseteq \{1,\dots, k\} \\ |S|= d+1 }} \frac{4^{d+1}}{d+2} \frac{\prod_{\substack{x,y \in S\\ x>y } } (x-y)}{\prod_{n=1}^d n!} .\]

So it remains to prove the second inequality. Trivially, we have \[ \prod_{\substack{x,y \in S\\ x>y } }(x-y) \leq  \prod_{\substack{x,y \in S\\ x>y }} k =  k^{d (d+1)/2}. \]Then the number of subsets of $\{1,\dots k\}$ of size $d+1$ is ${ k \choose d+1} \leq \frac{k ^{d+1}}{ (d+1)!}$. So we can take \[c_d = \frac { 4^{d+1}}{(d+2)! \prod_{n=1}^d n!}\] which gives $c_1=8/3$, $c_2 = 4/3$, $c_3= 16/ 45$, and $c_{d+1}/c_d = 4/ (d+3) (d+1)! \leq 1$ for $d \geq 1$. 

\end{proof}

\section{Dynamical Models of $\lambda$}

We are now ready to prove our main results.

\begin{theorem}\label{main1} For each multiplicative function $\lambda: \mathbb N \to \pm 1$ and natural number $d$, there exists a $d$-Fourier uniform dynamical model for $\lambda$ such that the only length $k$ sign patterns that appear in $(X, \mu, F, T)$ are those that occur in the image of the map $\sgn$ on the space of degree $d$ polynomials, and the number of distinct sign patterns that appear is $O(k^{(d+1)(d+2)/2})$. \end{theorem}

\begin{proof}  Let $\mathbb R^{d+1}$ be the space of polynomials of degree $d$, $H$ the lattice in $\mathbb R^{d+1}$ of polynomials that take even integer values on integer arguments, and $\mathbb R^{d+1}/H$ the quotient torus. 

Let $X = (\mathbb R^{d+1}/H) \times \hat{\mathbb Z}$. Let $\mu$ be the Haar measure on $X$. Let $F: X \to \pm 1$ send a pair $(f(x),n)$ of a polynomial and an element of $\hat{\mathbb Z}$ to $(-1)^{\lfloor f(0) \rfloor } $. Let $T$ send a pair $(f(x),n)$ to $(f(x+1),n+1)$. Let $M$ send a pair $(f(x),n)$ to $n$. For any natural number $a$  let $I_{a}$ send a pair $(f(x),n)$ where $a$ divides $n$ to $(f(ax) + \frac{1+\lambda(a)}{2}, \frac{n}{a})$.

It is easy to verify that $T$ is measure-preserving because it is a group automorphism on the first factor, trivial on the second, and translation by a group element on the third, all of which preserve the Haar measure. It is also easy to verify that $M$ preserves the measure.

Similarly, it is easy to verify all the algebraic identities among $M,T,I_{ab},F$ - in particular note that  $\frac{1+\lambda(a)}{2}$ is an integer and $(-1)^{ \frac{1+\lambda(a)}{2}} = \lambda(a)$.     

To prove $d$-Fourier uniformity, we must show that \[ \lim_{H \to \infty} \mathbb E_{1 \leq h_1,\dots, h_d \leq H} \mathbb E_{f \in \mathbb R^{d+1}/H} \prod_{ \delta_1,\dots, \delta_d \in \{0,1\}}  (-1)^ {\lfloor f (\sum_{i=1}^d \delta_i h_i) \rfloor} =0 .\] To do this, using Fourier series we can write $(-1)^{\lfloor f(n) \rfloor }$ as a linear combination of terms $e^{ 2\pi i \alpha f(n)} $ where $\alpha$ is a half-integer and in particular is nonzero. So the whole product is a sum of terms of the form \[e^{ 2\pi i  \sum_{\delta_1,\dots,\delta_d \in \{0,1\}}  \alpha_{\delta_1,\dots,\delta_d} f \left(\sum_{i=1}^d \delta_i h_i  \right)}.\] The expectation of this term over the space of polynomials vanishes unless \[\sum_{\delta_1,\dots,\delta_d \in \{0,1\}}  \alpha_{\delta_1,\dots,\delta_d} f \left(\sum_{i=1}^d \delta_i h_i\right)  \] is identically zero for all polynomials $f$, and otherwise it is $1$.  Then the expectation of this term over $h_i$ vanishes unless  $\sum_{\delta_1,\dots,\delta_d \in \{0,1\}}  \alpha_{\delta_1,\dots,\delta_d} f (\sum_{i=1}^d \delta_i h_i)  $ is identically zero for all $h_1,\dots,h_k$, as otherwise it can only be identically zero for a density zero fraction of $h_1,\dots,h_k$. To show this is impossible, and thus the expectations of all the terms vanish, it suffices to show that for some $f$, the terms $f (\sum_{i=1}^d \delta_i h_i) $ are linearly independent polynomials in the $h_i$ for distinct $\delta_i$. To do this, choose $f$ with all coefficients nonzero, and observe that the coefficient of $\prod_i h_i^{\delta_i}$ in $f (\sum_{i=1}^d \delta_i' h_i) $ is zero unless $\delta_i \leq \delta_i'$ for all $i$ and is nonzero if $\delta_i = \delta_i'$ for all $i$, so these coefficients form an upper-triangular matrix with nonzero diagonal, hence an invertible matrix. The same argument works with added congruence conditions.

The fact that the only sign patterns that occur are those in the definition of $\sgn(f)$ follows because the definitions of $F$ and $T$ match the definition of $\sgn$ exactly. The bound for the number of such sign patterns follows from Theorem \ref{mainbound}.

\end{proof} 

In the case $d=1$, the dynamical system we construct here matches the dynamical system $(\mathbb T^2, m_{\mathbb T^2}, T)$ defined in \cite[p. 7]{fhsuper} and also studied there as a counterexample for the Liouville function.

Furthermore, for each $\lambda$, $d$, and sign pattern $\epsilon_1,\dots,\epsilon_k$ with $k$ sufficiently large with respect to $d$, we will demonstrate the existence of a dynamical model for $\lambda$ which is $d$-Fourier uniform but in which the sign pattern $\epsilon_1,\dots,\epsilon_k$ does not appear.

Specifically we will prove:

\begin{theorem}\label{main2} For each multiplicative function $\lambda: \mathbb N \to \pm 1$, natural number $d$, and sign pattern $\epsilon_1 ,\dots,\epsilon_k$ such that there exists $q \leq k$ with $\phi(q) \geq \left( \frac{(d+1) (d+2)}{2}+ 2\right) \log_2 (q) +\log_2 (c_d)$, there exists a dynamical model for $\lambda$ that is $d$-Fourier uniform but in which the sign pattern $\epsilon_1,\dots,\epsilon_k$ does not appear. (Here $c_d$ is as in Theorem~\ref{mainbound}.)

\end{theorem}

Because $c_d \leq 1$ for $d \leq 3$, we may ignore that term for large $d$, and take $q$ asymptotic to $d^2 \log d$. We would not obtain an asymptotic here if we we did not keep track of the dependence on $d$ in Section 2.

The dynamical system in this proof is an extension of an idea due to Joni Ter\"{a}v\"{a}inen, which was essentially the $d=0$ case.

\begin{proof} Take such a $d,k,q$. Without loss of generality, we may assume $q=k$, because if $\epsilon_1,\dots,\epsilon_q$ does not appear then neither does $\epsilon_1,\dots,\epsilon_q,\dots,\epsilon_k$.

Let $\rho: \left(\mathbb Z/q\mathbb Z\right)^\times \to \pm 1$ be a function. We will construct a dynamical model for $\lambda$ associated to each such $\rho$. We will check that they satisfy the axioms and that they are $d$-Fourier uniform. Then we will check that for at least one such $\rho$, the sign pattern $\epsilon_1,\dots,\epsilon_q$ does not appear.

Let $\overline{\rho}$ be the unique function on $\hat{\mathbb Z}$ satisfying $\overline{\rho}(n)=\rho(n)$ if $n$ is invertible modulo $q$, $\overline{\rho}(pn)=\overline{\rho}(n)$ for all primes $p$ dividing $q$, and $\overline{\rho}(n)=1$ if the projection of $n$ to $\mathbb Z_p$ is $0$ for some $p$ dividing $q$.

 Let $\mathbb R^{d+1}$ be the space of polynomials of degree $d$, $H$ the lattice in $\mathbb R^{d+1}$ of polynomials that take even integer values on integer arguments, and $\mathbb R^{d+1}/H$ the quotient torus. 

Let $X = (\mathbb R^{d+1}/H) \times (\mathbb Z/q)^\times \times \hat{\mathbb Z}$. Let $\mu$ be the Haar measure on $X$. Let $F: X \to \pm 1$ send a tuple $(f(x),t,n)$ of a polynomial, a residue class mod $q$, and an element of $\hat{\mathbb Z}$ to $(-1)^{\lfloor f(0) \rfloor } \overline{\rho}(tn)$. Let $T$ send a tuple $(f(x),t,n)$ to $(f(x+1),t,n+1)$. Let $M$ send a tuple $(f(x),t,n)$ to $n$. For any natural numbers $a,b$ where $a$ is relatively prime to $q$ and $b$ is a product of primes dividing $q$, let $I_{ab}$ send a tuple $(f(x),t,n)$ where $ab$ divides $n$ to $(f(abx) + \frac{1+\lambda(ab)}{2}, a^{-1} t, \frac{n}{ab})$.

It is easy to verify that $T$ is measure-preserving because it is a group automorphism on the first factor, trivial on the second, and translation by a group element on the third, all of which preserve the Haar measure. It is also easy to verify that $M$ preserves the measure.

Similarly, it is easy to verify all the algebraic identities among $M,T,I_{ab},F$ - in particular note that  $\frac{1+\lambda(ab)}{2}$ is an integer and $(-1)^{ \frac{1+\lambda(ab)}{2}} = \lambda(ab)$.     

The proof of $d$-Fourier uniformity is the same as in the proof of Theorem~\ref{main1}, except that we have an additional periodic term, but this term does not modify the averaging over polynomials and so does not affect the argument.

Suppose a sign pattern $\epsilon_1,\dots,\epsilon_q$ appears in this dynamical model.  Then there exists a polynomial $f$, $t \in \mathbb Z/q\mathbb Z$, and $n \in \hat{\mathbb Z}$ such that for all $i$ from $1$ to $q$, $\epsilon_i = (-1)^{\lfloor f(i) \rfloor } \overline{\rho}(t(n+i))$. Hence there exist a polynomial $f$, $t \in \mathbb Z/q\mathbb Z$, and $n$ in $\mathbb Z/q$ such that for all $i$ from $1$ to $q$ with $n+i$ invertible mod $\mathbb Z$,  $\epsilon_i =  (-1)^{\lfloor f(i) \rfloor }\rho(t(n+i))$.  Hence $\rho$ is uniquely determined by $\epsilon_1,\dots,\epsilon_q$, the sign pattern $(-1)^{\lfloor f(i) \rfloor }$ for $n$ from $1$ to $q$, $t$, and $n$ mod $q$.  Hence the number of possible values of $\rho$ for which $\epsilon_1,\dots,\epsilon_q$ appears in the associated dynamical system is at most $ q \phi(q)   c_d q^{(d+1)(d+2)/2} $. If this is $<  2^{\phi(q)}$, then there exists a value of $\rho$ such that $\epsilon_1,\dots,\epsilon_q$ appears. Taking logs, we get the stated inequality.

\end{proof}

\begin{remark} The dynamical systems constructed here have other nice properties, which can be checked directly from the definitions. In particular, they satisfy the $k$-point Chowla conjecture for all $k\leq d+1$ and $k$ odd.  For $k\leq d+1$, this is because the values of a polynomial at any $d+1$ points are independent, and for $k$ odd, this is because the dynamical systems in question are isomorphic to those obtained by negating $F$, with the isomorphism given by adding $1$ to $f$. \end{remark}

\begin{remark}We can also show in certain cases that the Chowla conjecture fails for these dynamical systems at the first possible point. For example, we can show that for $d+2$ a power of two, the dynamical system of Theorem \ref{main1} does not satisfy $d+2$-point Chowla. Indeed

\[ \mathbb E_{x \in X} \prod_{t=1}^{d+2} F(T^tx)  = \mathbb E_{ f\in \mathbb R^{d+1}/H} \prod_{t=1}^{d+2} (-1)^{ \lfloor f(t) \rfloor }  =  \mathbb E_{ f\in \mathbb R^{d+2}/H} \prod_{t=1}^{d+2} \sum_{\substack {\alpha \in \mathbb Z \\ k \textrm{ odd} }} \frac{ 2 e^{ \pi i \alpha f(t)}}{\pi k} \] \[= \frac{2^{d+2}}{\pi^{d+2}} \sum_{\substack {\alpha_1,\dots, \alpha_{d+2}  \in \mathbb Z \\ \alpha_t \textrm{ odd} }}   \frac{1}{\prod_{t=1}^{d+2} \alpha_t} \mathbb E_{ f\in \mathbb R^{d+1}/H} e^{ \pi i \sum_{t=1}^{d+2} \alpha_t f(t)}.\]

Now this expectation vanishes unless $\sum_{t=1}^{d+2} \alpha_t f(t)$ is identically zero, in which case it is $1$.  Because the space of polynomials of degree $d+1$ is codimension one inside the space of all functions on $\{1,\dots,d+2\}$, this is identically zero if and only if $\alpha_t$ is proportional to $(-1)^t {d+1 \choose t-1}$, which gives the unique linear form that vanishes on this space. Thus we set $\alpha_t = \alpha (-1)^t {d+1 \choose t-1}$. Because $d+2$ is a power of $2$, these binomial coefficients are odd, so we have all $\alpha_t$ odd integers if and only if $\alpha$ is an odd integer, and thus the sum is

\[  \frac{2^{d+2}}{\pi^{d+2}}  \sum_{ \substack{ \alpha \in \mathbb Z\\ \alpha\textrm{ odd}}} (-1)^{ \frac{d+2}{2}} \frac{1}{\alpha^{d+2} \prod_{i=0}^{d+1} { d+1 \choose i}}  =  \frac{ 2^{d+3} (-1)^{ \frac{d+2}{2}}  \zeta(d+2)}{ \pi^{d+2}  (1- 2^{-d-2}) \prod_{i=0}^{d+1} { d+1 \choose i}} \neq 0 .\]
 For instance, if $d=2$ this produces a dynamical system that satisfies 1-Fourier uniformity, 2-point Chowla, and odd-point Chowla but not 4-point Chowla. \end{remark}

\begin{remark} The relation of dynamical models for $\lambda$ to the axioms of \cite[Proposition 1]{tfurstenburg} is as follows: 

Let $(X, \mu, T, M,I_n)$ be  a dynamical model for $\lambda$ the standard Liouville function. Consider the inverse limit $(X',\mu')$ of the system of measurable spaces $( M^{-1}( n \mathbb Z), n \mu)$ under the measure-preserving transformations $I_m: M^{-1}( nm \mathbb Z) \to M^{-1}(n \mathbb Z)$. Let the transformation $\phi(n) = an+b$ act on  $X'$ by the limit of the maps $M^{-1} ( an \mathbb Z) \to M^{-1} (n \mathbb Z)$   defined by $ T^{nb} $ composed with the obvious inclusion $M^{-1} ( an \mathbb Z) \subseteq M^{-1} (n \mathbb Z)$. We can pul back $F$ and $M$ to functions on $X'$.

Then the data of $(X',\mu',F,M)$ plus the action of $\phi$ satisfy axioms (ii), (iii), and (iv) of Tao's Proposition 1, and may or may not satisfy axiom (i). 

The reason that the inverse limit construction is required to pass between the two sets of axioms is as follows. A dynamical model of $\lambda$ ``keeps track" of the nearby values $\dots, \lambda(n-1), \lambda(n), \lambda(n+1),\lambda(n+2),\dots$ as well as the remainder of $n$ modulo another natural number $m$. Using this, say if $n \equiv a \mod m$, it is possible to find also the nearby values to $\lambda\left( \frac{n-a}{m}\right)$, as \[\lambda \left( \frac{n-a}{m} + k\right) = \frac{ \lambda(n-a + mk)}{\lambda(m)}.\] The map $I_m$ keeps track of this division, but it does not make the dynamical system any more complicated because the information is already contained in $\lambda$. On the other hand, the axioms of \cite{tfurstenburg} involve a multiplication-by-$m$ map, which means that the value of the dynamical system at a point congruent to $n$ must keep track of the values of $\lambda$ near $nm$ for each natural number $m$. Because this is more information, the simplest examples of dynamical systems satisfying the axioms lie on more complicated spaces. In particular, the inverse limit turns out to be just the right construction to add this extra data.\end{remark}

\section{Positive results}

In this section, we describe two simple arguments that complement Theorems \ref{main1} and \ref{main2} by proving existence results for an arbitrary sign pattern and lower bounding the total number of sign patterns that appear.   It will be convenient, rather than using assumptions on Fourier uniformity, to assume the following version of Chowla's conjecture. For some appropriate $k$, for all $c_1,\dots,c_k \in \mathbb N$, we have

\begin{equation}\label{chowla} \mathbb E_{x\in X} \prod_{i=1}^k F( T^{c_i} x) =0.\end{equation}

Afterwards, we will explain the relevance of Equation~\ref{chowla} and thus these results to dynamical models for Liouville.

\begin{proposition}\label{back2} Let $r\geq 1$ be a natural number and let $(X,\mu, F, T)$ be a dynamical system satisfying Equation~\ref{chowla} for all $1\leq k \leq 2 r  +1$. Then all sign patterns of length $2 r+2$ appear in $(X,\mu,F,T)$. \end{proposition}

This is a slight variant of an argument independently discovered by myself and Kaisa Matom\"{a}ki, which appeared in \cite[Proposition 7.1]{ttodd}.

\begin{proof} Let $k = 2 r+2$ and let $\epsilon_1,\dots, \epsilon_k$ be a sign pattern of length $k$. Then the measure of the set of $x$ where $F(T^{i} x) = \epsilon_i$ for all $i$ from $1$ to $k$ is \[ \mathbb   E_{x\in X} \frac{1}{2^k}  \sum_{ S \subseteq \{1,\dots,k\}} \prod_{i \in S} \epsilon_i F(T^i x) = \frac{1}{2^k}  \sum_{ S \subseteq \{1,\dots,k\}}  \mathbb E_{x \in X}\prod_{i \in S} \epsilon_i \mathbb E_{x \in X}  F(T^i x) .\] For any nonempty $S$ other than $\{1,\dots,k\}$, $0 < |S|  \leq 2  r +1$ and so by assumption the expectation vanishes. Because the contribution of the empty set is $1$, it suffices to show that  $|\mathbb E_{x \in X}\prod_{i \in \{1,\dots,k\}} F(T^i x) | <1$. In fact, we will show it is at most $1/2$, because 

\[ 2|\mathbb E_{x \in X}\prod_{i \in \{1,\dots,k\}} F(T^i x) | = | \mathbb E_{x \in X}\prod_{i \in \{1,\dots,k\}} F(T^i x)  + \mathbb E_{x \in X}\prod_{i \in \{1,\dots,k\}}F(T^{i+1} x)  | \] \[  \leq \mathbb E_{x\in X} | \prod_{i \in \{1,\dots,k\}} F(T^i x)  + \prod_{i \in \{1,\dots,k\}} F(T^{i+1} x)|  = \mathbb E_{x \in X}  (1 + \prod_{i \in 1 ,\dots, k} F(T^{i} x) F(T^{i+1} x) ) \] \[= \mathbb E_{x\in X} 1 + \mathbb E_{x\in X} F(T x) F(T^{k+1} x) =1\] using on the second and third lines that $F$ is $\pm 1$ valued and, in the last step, Equation~\ref{chowla}. \end{proof}

\begin{proposition}\label{back1} Let $d \geq 0$ be a natural number number and let $(X,\mu, F, T)$ be a dynamical system satisfying Equation~\ref{chowla} for all $1\leq k \leq 2 r +1$. Then the number of $m$-term sign patterns appearing in $(X, \mu, F, T)$ is at least  $\frac{ 2 m^{r}}{(2r-1)!!} $.\end{proposition}

\begin{proof} Let $\epsilon_1,\dots,\epsilon_m$ be a sign pattern of length $m$. It will suffice to show that the probability that $F(T^i x) = \epsilon_i$ for all $i$ from $1$ to $m$ is at most $\frac{ (2r-1)!!}{ 2 m^{r}}$, because then the total number of sign patterns with positive expectation must be at least $\frac{ 2 m^{r}}{(2r-1)!!} $ .

To do this, observe that $( \sum_{i=1}^m \epsilon_i F(T^i x) )^{2r}  ( m + \sum_{i=1}^m \epsilon_i F( T^i x)) \geq 0$, and equals $2m^{2 r+1}$ if $F(T^i x) =\epsilon_i$ for all $i$, so this probability is at most \[ \frac{1}{ 2m^{2 r+1}} \mathbb E_{x \in X} ( \sum_{i=1}^m \epsilon_i F(T^i x) )^{2r}  ( m + \sum_{i=1}^m \epsilon_i F( T^i x)) 
\] \[ =  \frac{1}{ 2m^{2 r+1}}  \left( m \sum_{ 1 \leq i_1,\dots,  i_{ 2 r}\leq m} \prod_{j=1}^{ 2 \lceil \frac{d}{2}} \epsilon_{i_j} F(T^{i_j} x) + \sum_{ 1 \leq i_1,\dots,  i_{ 2 r+1}\leq m} \prod_{j=1}^{ 2 \lceil \frac{d}{2}+1} \epsilon_{i_j} F(T^{i_j} x) \right).\]

By assumption, all these expectations vanish unless each $i_j$ appears with even multiplicity, which can only happen in the first term, and in this case the expectation is one. The number of tuples with even multiplicity is at most the number of ways of dividing  $1,\dots,2 r$ into sets of size two times $m^{r}$, which is $(2 r -1)!! m^{ r}$, and the total contribution of these tuples is at most $(2 r -1)!! m^{ r+1}$, so we get the stated bound. \end{proof}

Let $\lambda$ be a $\pm1$-valued function multiplicative function that does not weakly pretend to be a Dirichlet character in the sense of \cite[p. 3]{ttodd} (e.g. the Liouville function). Let $(X,\mu,F, T, M,I_n)$ be a dynamical model for $\lambda$ that is $d$-Fourier uniform. As mentioned earlier, it should be possible to check, by a variant of \cite[Section 3]{tchowlasarnak}, that for all $0< k \leq d+1$ and $c_1,\dots,c_k \in \mathbb N$,   This is the dynamical analogue of the Chowla conjecture - specifically, the version using only shifts and not the more general version with linear forms. The version with shifts should apply equally well but is notationally more complicated and is not necessary for this section. By an argument along the lines of \cite{ttodd}, it is possible to show Equation \ref{chowla} for all $k$ odd. Combining these, we see that Equation \ref{chowla} holds for all $k \leq 2 \lceil \frac{d}{2} \rceil +1$. So we can apply Propositions~\ref{back2} and~\ref{back1} with $r= \lceil \frac{d}{2} \rceil$.  However, checking this formally would require going over those arguments, which are more difficult than anything in this paper, line-by-line, so we avoid doing that here. 

Similarly, in the case $d=1$, we believe that the lower bound of Proposition~\ref{back1} can be improved for dynamical models of a strongly aperiodic multiplicative function $\lambda$ by following \cite[with ``strongly aperiodic" defined in Definition 2.9]{fhsystem}. Theorem 1.2 of that paper shows that the number of sign patterns is superlinear for the multiplicative function itself. To adapt it to dynamical models, one mainly has to adapt the arguments of \cite[Section 3]{fhsystem}, which are based on \cite{ttodd}. Using \cite[Theorem 3.8]{fhsystem}, one can show that a dynamical model for $\lambda$ is a factor of the system of arithmetic progressions with prime steps of itself in the sense of \cite[Definition 4.1]{fhsystem}. The remaining arguments are dynamical in nature.

Note that the lower bound on Proposition \ref{back1} on the number of sign patterns is $m$ raised to something linear in $d$, while the upper bound (in one case) of Theorem \ref{main1} is $m$ raised to something quadratic in $d$. Similarly, Proposition \ref{back2} guarantees the existence of sign patterns up to a length that is linear in $d$, while Proposition \ref{back1} shows nonexistence (in one case) of any sign pattern of length $d^2 \log d$. So in each case the gap between results and counterexamples is roughly quadratic. Perhaps arguments using sophisticated tools of dynamical systems theory can improve on the easy arguments of this section and close that gap.

\bibliographystyle{plain}
\bibliography{references}

\end{document}